\documentclass[11pt,oneside,english]{amsart}
\usepackage[utf8]{inputenc}
\usepackage{amstext}
\usepackage{amsthm}
\usepackage{amssymb}
\usepackage{graphicx}

\makeatletter

\providecommand{\tabularnewline}{\\}

\numberwithin{equation}{section}
\numberwithin{figure}{section}
\theoremstyle{plain}
\newtheorem{thm}{\protect\theoremname}[section]
  \theoremstyle{definition}
  \newtheorem{problem}[thm]{\protect\problemname}
  \theoremstyle{plain}
  \newtheorem{conjecture}[thm]{\protect\conjecturename}
  \theoremstyle{plain}
  \newtheorem{lem}[thm]{\protect\lemmaname}
  \theoremstyle{definition}
  \newtheorem{defn}[thm]{\protect\definitionname}
  \theoremstyle{plain}
  \newtheorem{prop}[thm]{\protect\propositionname}
  \theoremstyle{plain}
  \newtheorem{fact}[thm]{\protect\factname}

\usepackage{multicol}

\makeatother

\usepackage{babel}
  \providecommand{\conjecturename}{Conjecture}
  \providecommand{\definitionname}{Definition}
  \providecommand{\factname}{Fact}
  \providecommand{\lemmaname}{Lemma}
  \providecommand{\problemname}{Problem}
  \providecommand{\propositionname}{Proposition}
\providecommand{\theoremname}{Theorem}

\begin{document}
\global\long\def\Z{\mathbb{\mathbf{Z}}}
\global\long\def\Spec{\mathrm{Spec}}
\global\long\def\G{\mathbb{G}}
\global\long\def\X{\mathbb{X}}
\global\long\def\GL{\mathrm{GL}}
\global\long\def\O{\mathcal{O}}
\global\long\def\SL{\mathrm{SL}}
\global\long\def\C{\mathbf{C}}
\global\long\def\tr{\mathrm{tr}}
\global\long\def\Aut{\mathrm{Aut}}
\global\long\def\Out{\mathrm{\mathrm{Out}}}
\global\long\def\FF{\mathbb{F}}
\global\long\def\Q{\mathbf{Q}}
\global\long\def\R{\mathbf{R}}
\global\long\def\Perms{\mathrm{Perms}}
\global\long\def\Fib{\mathrm{Fib}}
\global\long\def\F{\mathbf{F}}

\global\long\def\PGL{\mathrm{PGL}}
\global\long\def\Y{\mathbb{Y}}
\global\long\def\Aut{\mathrm{Aut}}
\global\long\def\sign{\mathrm{sign}}
\global\long\def\PSL{\mathrm{PSL}}
\global\long\def\trace{\mathrm{trace}}

\title{The cycle structure of a Markoff automorphism over finite fields}

\thanks{M. Magee was supported in part by N.S.F. award DMS-1701357. All authors
were supported in part by Sam Payne's N.S.F. CAREER award DMS\textendash 1149054. }

\author{Alois Cerbu, Elijah Gunther, Michael Magee, Luke Peilen}
\begin{abstract}
We begin an investigation of the action of pseudo-Anosov elements
of $\Out(\F_{2})$ on the Markoff-type varieties 
\[
\X_{\kappa}:\:x^{2}+y^{2}+z^{2}=xyz+2+\kappa
\]
 over finite fields $\FF_{p}$ with $p$ prime. We first make a precise
conjecture about the permutation group generated by $\Out(\F_{2})$
on $\X_{-2}(\FF_{p})$ that shows there is no obstruction at the level
of the permutation group to a pseudo-Anosov acting `generically'.
We prove that this conjecture is sharp. We show that for a fixed pseudo-Anosov
$g\in\Out(\F_{2})$, there is always an orbit of $g$ of length $\geq C\log p+O(1)$
on $\X_{\kappa}(\FF_{p})$ where $C>0$ is given in terms of the eigenvalues
of $g$ viewed as an element of $\GL_{2}(\Z)$. This improves on a
result of Silverman from \cite{Silverman} that applies to general
morphisms of quasi-projective varieties. We have discovered that the
asymptotic $(p\to\infty)$ behavior of the longest orbit of a fixed
pseudo-Anosov $g$ acting on $\X_{-2}(\FF_{p})$ is dictated by a
dichotomy that we describe both in combinatorial terms and in algebraic
terms related to Gauss's ambiguous binary quadratic forms, following
Sarnak \cite{Sarnak}. This dichotomy is illustrated with numerics,
based on which we formulate a precise conjecture in Conjecture \ref{conj:behaviour-of-PA}.
\end{abstract}

\maketitle

\section{Introduction}

For $\kappa\in\Z,$ let $\X_{\kappa}$ denote the affine surface 
\begin{equation}
\X_{\kappa}:\:x^{2}+y^{2}+z^{2}=xyz+2+\kappa.\label{eq:Xkdef}
\end{equation}
When $\kappa=-2$, $\X_{-2}$ is Markoff's surface. A theorem of Markoff
\cite{Markoff2} relates the integer points on $\X_{-2}$ to the Diophantine
properties of $\Q$; in particular to the Markoff spectrum. In a different
vein, the real and complex points of $\X_{\kappa}$ are related to
moduli spaces of $\SL_{2}(\C)$-local systems on a torus with one
puncture \cite{Goldman}. Due to this connection, letting $\F_{2}$
denote the free group on 2 generators, the group $\Out(\F_{2})\cong\GL_{2}(\Z)$
acts by automorphisms of $\X_{\kappa}$, viewed as a scheme of finite
type over $\Z$. The group $\Out(\F_{2})$ is the mapping class group
of the torus with one puncture, and the free group $\F_{2}$ is the
fundamental group of this surface. As such, $\Out(\F_{2})$ is subject
to Thurston's classification of mapping class group elements \cite{THURSTON}
into periodic, reducible, or \emph{pseudo-Anosov} (p-A.) elements.
From the point of view of $\GL_{2}(\Z),$ an element is p-A. if it
is \emph{hyperbolic}, that is, has two distinct real eigenvalues.
The current paper aims to investigate how p-A.\emph{ }elements of
$\Out(\F_{2})$ act on $\X_{\kappa}(\FF_{p})$ for prime $p$. 

The study of p-A.\emph{ }elements of $\Out(\F_{2})$ acting on $\X_{\kappa}(\R)$
and $\X_{\kappa}(\C)$ has been ongoing since the early 1980s, instigated
by a paper of Kohmoto, Kadanoff and Tang \cite{KKT} where the spectrum
of a 1D lattice Schrödinger operator with a quasiperiodic potential
was related to the dynamics of a particular p-A. automorphism (the
\emph{Fibonacci substitution})\emph{ }on $\X_{\kappa}(\R)$. In \cite{Cantat},
Cantat resolved a conjecture of Kadanoff relating the topological
entropy of a p-A. element acting on $\X_{\kappa}(\R)$ to the largest
eigenvalue of the corresponding matrix in $\GL_{2}(\Z)$. See also
Bowditch \cite{BOW} for some related questions.

Here, we begin a parallel study for the action of p-A. elements on
$\X_{\kappa}(\FF_{p}).$ Any p-A. element $\Phi$ of $\Out(\F_{2})$
gives for each prime $p$ a permutation $\Phi_{p}$ of $\X_{\kappa}(\FF_{p})$.
In this paper we propose that in the study of p-A. $\Phi$ acting
on $\X_{\kappa}(\FF_{p})$, one should replace topological entropy
by the asymptotic complexity of the family of permutations $\{\Phi_{p}\}$.
In particular, we ask the following question. 
\begin{problem}
\emph{\label{prob:main-problem}For fixed p-A. $\Phi$, what is the
asymptotic behavior of 
\[
\frac{\log(\:\text{{\rm longest cycle of} \ensuremath{\Phi_{p}\:\mathrm{on}\:\X_{\kappa}(\FF_{p})\:)}}}{\log p}
\]
as $p\to\infty$?}
\end{problem}
Empirically, the answer to this question is quite surprising (see
Conjecture \ref{conj:behaviour-of-PA} below). We also obtain a theoretical
result towards this question in Theorem \ref{thm:long-cycles-uniform}
below. We will be primarily interested in the case of $\kappa=-2$,
although we prove some of our results for general $\kappa$.

Before tacking Problem \ref{prob:main-problem}, a preliminary question
intervenes. It could a priori be the case that the permutation group
generated by $\Out(\F_{2})$ on $\X_{\kappa}(\FF_{p})$ is highly
restricted and this would of course affect how a single element can
behave. 

Let $\X_{-2}^{*}(\FF_{p})=\X_{-2}(\FF_{p})-(0,0,0)$. Bourgain, Gamburd,
and Sarnak prove in \cite[Theorem 2]{BGS} that for all primes outside
a very small exceptional set, the action of $\Out(\F_{2})$ on $\X_{-2}^{*}(\FF_{p})$
is transitive, which was a conjecture of McCullough and Wanderley
from \cite{Transitive}. Sarnak has raised more generally the question
of what permutation group is generated by the action of $\Out(\F_{2})$
on $\X_{-2}^{*}(\FF_{p})$. 

It follows from work of Horowitz \cite{Horowitz} (see also Goldman
\cite{Goldman}) that 
\[
\Aut(\X_{\kappa})\cong\mathrm{P}\GL_{2}(\Z)\ltimes N
\]
where the $\PGL_{2}(\Z)$ factor is induced by $\Out(\F_{2})$ and
$N$ is the Klein four-group generated by even sign changes
\[
n_{1}:(x,y,z)\longmapsto(x,-y,-z)
\]
(similarly $n_{2}$, $n_{3}$). For $p$ odd, each $N$-orbit on $\X_{-2}^{*}(\FF_{p})$
contains four distinct points (see Lemma \ref{no (0,0,z)'s} below).
Thus $\Out(\F_{2})$ cannot act 2-transitively on $\X_{-2}^{*}(\FF_{p})$
for any prime, since it must permute orbits of $N$. In light of this
observation, we should examine instead the action of $\Out(\F_{2})$
on the set of $N$-orbits in $\X_{-2}^{*}(\FF_{p})$, which we denote
by $\Y_{-2}(\FF_{p})$.

Let $H(p)$ denote the permutation group generated by $\Out(\F_{2})$
acting on $\Y_{-2}(\FF_{p})$. We write $A_{n}$ for the alternating
group on $n$ letters and $S_{n}$ for the symmetric group. We prove
the following.
\begin{thm}
\label{thm:containment} Let $n=|\Y_{-2}(\FF_{p})|$ and let $p>3$.
Then, $H(p)\leq A_{n}$ if and only if $p\equiv3\pmod{16}$. 
\end{thm}
The above theorem, alongside computations of $H(p)$ for $p\le47$,
lead us to conjecture the following:
\begin{conjecture}
\label{conj: 1} Let $H(p)$ denote the permutation group induced
by the action of $\Out(\F_{2})$ on $\Y_{-2}(\FF_{p})$, and let $n=|\Y_{-2}(\FF_{p})|$.
Then when $p>3$ 
\end{conjecture}
\begin{itemize}
\item $H(p)\cong S_{n}$ if $p\not\equiv3\pmod{16},$ 
\item $H(p)\cong A_{n}$ if $p\equiv3\pmod{16}.$ 
\end{itemize}
Meiri and Puder have proved in \cite{MP} that $H(p)$ contains $A_{n}$
whenever $p\equiv1\pmod4$ and $p$ is outside the Bourgain-Gamburd-Sarnak
exceptional set, and also for a density 1 set of primes without any
congruence condition. In these cases, Theorem \ref{thm:containment}
describes exactly what $H(p)$ is. This shows there is no obstruction
at the level of the group $H(p)$ to a p-A. element behaving `generically'
on $\Y_{-2}(\mathbb{F}_{p})$.

We now describe our theoretical result towards Problem \ref{prob:main-problem}.
In \cite{Silverman} Silverman studied a more general version of this
problem and obtained as a consequence the following result.
\begin{thm}[{Silverman \cite[Theorem 3(a)]{Silverman}}]
\label{thm:silverman}Let $K$ be a number field with ring of integers
$\O_{K}$ and $V$ a quasi-projective variety defined over $K$. Let
$\varphi:V\to V$ be a morphism defined over $K$. If $\mathfrak{p}\in\Spec(\O_{K})$
is such that $V$ and $\varphi$ have good reduction at $\mathfrak{p}$,
then write $\varphi_{\mathfrak{p}}$ and $V(\FF_{\mathfrak{p}})$
for these reductions and $N(\mathfrak{p})$ for the norm of this prime.
For any $\epsilon>0$, the set of $\mathfrak{p}\in\Spec(\O_{K})$
such that there is good reduction of $\varphi$ and $V$ at $\mathfrak{p}$,
and an orbit of $\varphi_{\mathfrak{p}}$ on $V(\FF_{\mathfrak{p}})$
of length $\geq(\log N(\mathfrak{p}))^{1-\epsilon}$ has analytic
density 1.
\end{thm}
This result applies directly to Problem \ref{prob:main-problem}.
Silverman's result actually provides many orbits of length $\geq(\log N(\mathfrak{p}))^{1-\epsilon}$.

What we can achieve in the current context is the removal of the $\epsilon$
from Theorem \ref{thm:silverman}, and get a statement for all primes
$p$ instead of just analytic density 1. Furthermore, our bounds are
independent of $\kappa$.
\begin{thm}
\label{thm:long-cycles-uniform}Given a pseudo-Anosov $g\in\Out(\F_{2})$,
let $\lambda$ denote the eigenvalue of largest modulus of the corresponding
matrix in $\GL_{2}(\Z)$. For any $\kappa\in\Z$, as $p\to\infty$,
$g$ has an orbit of length at least
\[
\frac{\log p}{\log|\lambda|}+O_{g}(1)
\]
on $\X_{\kappa}(\FF_{p})$. The implied constant depends on $g$,
but not on $\kappa$.
\end{thm}
Now we describe our numerical results which show what the answer to
Problem \ref{prob:main-problem} should be, at least for $\Y_{-2}(\FF_{p})$.
First we give two natural guesses, that turn out to both be wrong.

\textbf{Guess~1:~a~p-A.~$g$~acts~as~a~random~map~on~$\Y_{-2}(\FF_{p})$.}
A random map from a set of size $N$ to itself has with high probability
its longest orbit of size $\asymp\sqrt{N}$. Since $|\Y_{-2}(\FF_{p})|\asymp p^{2}$
this predicts the longest orbit of $g$ acting on $\Y_{-2}(\FF_{p})$
will have size $\asymp p$. This fact comes from a collision heuristic
based on the `Birthday paradox'. However, this heuristic is not convincing,
since $g$ is invertible, so should really be viewed as a random permutation
(for some notion of random, see next guess). Also, this guess doesn't
give the right answer in general (see below).

\textbf{Guess~2:~a~p-A.~$g$~acts~as~a~random~permutation~on~$\Y_{-2}(\FF_{p})$.
}Perhaps we should model the action of $g$ on $\Y_{-2}(\FF_{p})$
by a permutation chosen uniformly at random from $A_{n}$ or $S_{n}$,
where $n=|\Y_{-2}(\FF_{p})|$, according to Conjecture \ref{conj: 1}.
To simplify things, let us just consider $S_{n}$, the case of $A_{n}$
being similar. Then it is known that a permutation drawn uniformly
at random from $S_{n}$ has a cycle of length at least $n/2$ in its
cycle decomposition with positive probability. This fact is closely
related to the well-known `100 Prisoners Problem' posed in \cite{gal2007cell}.
So this would predict for fixed p-A. $g$ that as $p$ varies we should
often (in fact being more careful with the statistics, with high probability)
see an orbit of length $\asymp p^{2}$ of $g$ on $\Y_{-2}(\FF_{p})$.
This guess also turns out not to be correct in general.

\begin{center}
\begin{figure}
\begin{centering}
\includegraphics[clip,scale=0.7,clip,trim=2cm 14cm 5cm 2cm,]{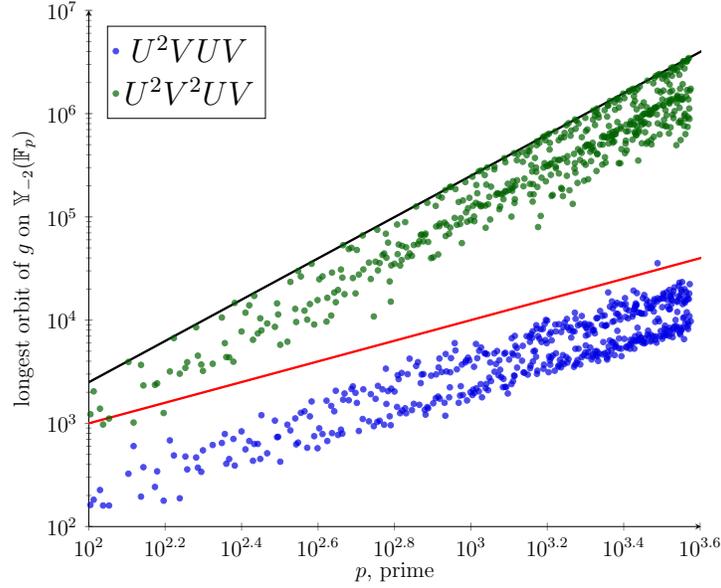}
\par\end{centering}
\caption{\label{fig:graphs-of-longest-orbit}This shows the longest orbits
of two p-A. elements $U^{2}VUV$ (blue) and $U^{2}V^{2}UV$ (green)
on $\protect\Y_{-2}(\protect\FF_{p})$. The black line is $y=p^{2}/4$,
which is asymptotic to $|\protect\Y_{-2}(\protect\FF_{p})|$. The
red line is $y=10p$. The plots are in log vs log scale axes with
$p$ on the horizontal axis and longest orbit on the vertical.}
\end{figure}
\par\end{center}

To describe our numerics, we introduce special elements of $\PGL_{2}(\Z)$.
We first note, if $g$ has determinant $-1$, then the qualitative
behavior of the longest orbit of $g$ will be governed by that of
$g^{2}$, which has determinant $1$. So it is sufficient (at least
for the phenomena we show) to consider only elements of $\PSL_{2}(\Z)$.
Let 
\[
U=\left(\begin{array}{cc}
1 & 0\\
1 & 1
\end{array}\right),\quad V=\left(\begin{array}{cc}
1 & 1\\
0 & 1
\end{array}\right).
\]

\begin{lem}
\label{lem:UV forms}Every hyperbolic element of $\PSL_{2}(\Z)$ is
conjugate to an element
\[
U^{n_{1}}V^{m_{1}}\ldots U^{n_{k}}V^{m_{k}}
\]
 with $k>0$ and all $n_{i},m_{i}>0$. We call this a reduced $UV$-word.
\end{lem}
This is well-known, but for completeness we prove Lemma \ref{lem:UV forms}
in Section \ref{sec:Algebraic-Characterization-of} below. Since the
orbit lengths of $g$ on $\X_{\kappa}(\FF_{p})$ are the same after
conjugation, we may simply consider reduced $UV$-words in what follows.
Our conjectural answer to Problem \ref{prob:main-problem} is based
on a dichotomy for hyperbolic $g\in\PSL_{2}(\Z)$.
\begin{defn}
A reduced $UV$-word $U^{n_{1}}V^{m_{1}}\ldots U^{n_{k}}V^{m_{k}}$
is a \emph{cyclic palindrome }if its reverse can be cyclically rotated
to obtain the original word. For example:
\[
U^{2}VUV\xrightarrow{\mathrm{reverse}}VUVU^{2}\xrightarrow{\mathrm{rotate}}UVU^{2}V\xrightarrow{\mathrm{rotate}}VU^{2}VU\xrightarrow{\mathrm{rotate}}U^{2}VUV.
\]
\end{defn}
Then $U^{2}VUV$ is a cyclic palindrome, whereas $U^{2}V^{2}UV$ is
not. Following Sarnak \cite{Sarnak} (who follows terminology of Gauss)
we make the following definition.
\begin{defn}
Say $g\in\PSL_{2}(\Z)$ is \emph{ambiguous }if the conjugacy class
of $g$ in $\PSL_{2}(\Z)$ is conjugated to the conjugacy class of
$g^{-1}$ in $\PSL_{2}(\Z)$ by an element of $\PGL_{2}(\Z)$ of determinant
$-1$.
\end{defn}
Our two definitions actually coincide.
\begin{prop}
\label{prop:two-defs}Let hyperbolic $g\in\PSL_{2}(\Z)$ be given
by a reduced $UV$-word. Then the $UV$-word is a cyclic palindrome
if and only if $g$ is ambiguous.
\end{prop}
We prove Proposition \ref{prop:two-defs} in Section \ref{sec:Algebraic-Characterization-of}.
In Figure \ref{fig:graphs-of-longest-orbit} we show the longest orbits
of $U^{2}VUV$ and $U^{2}V^{2}UV$ on $\Y_{-2}(\FF_{p})$. They evidently
have strikingly different behaviors. Note that
\[
U^{2}VUV=\left(\begin{array}{cc}
2 & 3\\
5 & 8
\end{array}\right),\quad U^{2}V^{2}UV=\left(\begin{array}{cc}
3 & 5\\
7 & 12
\end{array}\right),
\]
so they are both hyperbolic. However, Figure \ref{fig:graphs-of-longest-orbit}
shows that the longest orbit of $U^{2}VUV$ is on the order of $p$
and that of $U^{2}V^{2}UV$ is on the order of $p^{2}$. Based on
further evidence (see Table \ref{tab:This-table-gives-evidence} and
Figure \ref{fig:Histograms-showing-the}), we are led to conjecture
that the crucial difference between these words is that $U^{2}VUV$
is a cyclic palindrome/ambiguous. Write $L(g;p)$ for the longest
orbit of $g$ on $\Y_{-2}(\FF_{p})$. We make the following conjecture:

\newpage{}
\begin{conjecture}
\label{conj:behaviour-of-PA}Let $g\in\PSL_{2}(\Z)$ be hyperbolic.
If $g$ is ambiguous then
\begin{enumerate}
\item There are constants $C_{1}=C_{1}(g)>0$ and $C_{2}=C_{2}(g)>C_{1}$
such that $C_{1}p\leq L(g;p)\leq C_{2}p$ for all primes $p$.
\item The discrete probability measures 
\[
\frac{1}{\#\{\text{primes \ensuremath{p} \ensuremath{\leq}\ensuremath{X}\}}}\left(\sum_{p\leq X}\delta_{\frac{L(g;p)}{p}}\right)
\]
converge as $X\to\infty$ to a compactly supported Borel probability
measure on $\R$.
\end{enumerate}
If $g$ is \textbf{not }ambiguous then
\begin{enumerate}
\item There is a constant $c=c(g)$ such that $L(g;p)\geq cp^{2}$ for all
primes $p$.
\item The discrete probability measures 
\[
\frac{1}{\#\{\text{primes \ensuremath{p} \ensuremath{\leq}\ensuremath{X}\}}}\left(\sum_{p\leq X}\delta_{\frac{L(g;p)}{p^{2}}}\right)
\]
converge as $X\to\infty$ to a compactly supported Borel probability
measure on $\R$.
\end{enumerate}
As a particular consequence, we conjecture that the answer to Problem
\ref{prob:main-problem} for $\kappa=-2$ is
\[
\lim_{p\to\infty}\frac{\log(L(g;p))}{\log p}=\begin{cases}
1 & \text{if \ensuremath{g} is ambiguous.}\\
2 & \text{if \ensuremath{g} is not ambiguous.}
\end{cases}
\]
\end{conjecture}
\begin{center}
\begin{figure}
\begin{centering}
\includegraphics[scale=0.9,clip,trim=2.5cm 18cm 2cm 2cm,]{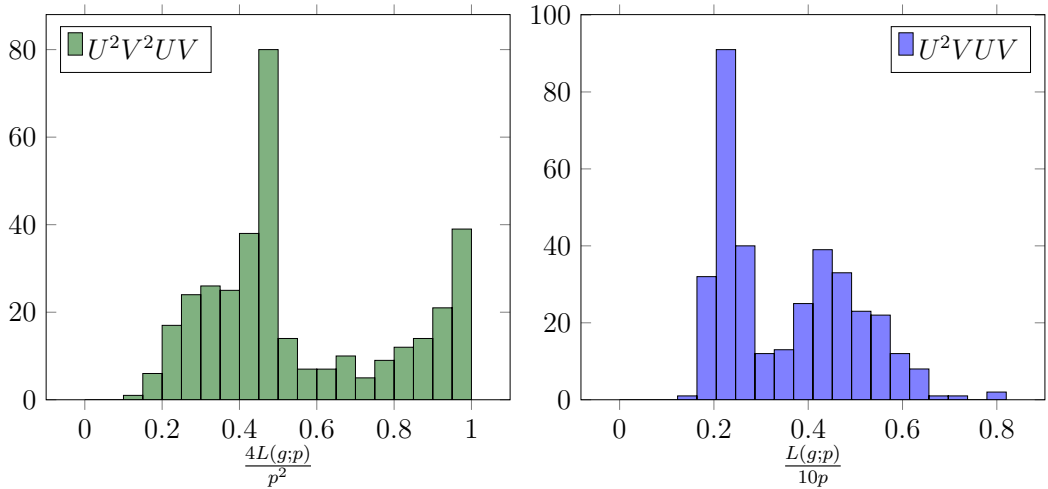}
\par\end{centering}
\caption{\label{fig:Histograms-showing-the}Histograms showing the distributions
that feature in Conjecture \ref{conj:behaviour-of-PA}. Here $p$
ranges between 1009 and 3761. For $U^{2}V^{2}UV$ the histogram shows
the distribution of $L(g;p)(p^{2}/4)^{-1}$. Note $|\protect\Y_{-2}(\protect\FF_{p})|$
is asymptotic to $p^{2}/4$. For $U^{2}VUV$ the distribution is of
$L(g;p)(10p)^{-1}$. The $10$ is not significant and has just been
chosen to scale the data. One outlier ($p=3079$, $L(g;p)=35585$)
has been removed from the $U^{2}VUV$ chart.}
\end{figure}
\par\end{center}

\renewcommand{\arraystretch}{1.2}

\begin{table}
\begin{tabular}{|cr|}
\hline 
$g$ \textbf{not }ambiguous & \tabularnewline
\hline 
\multicolumn{1}{|c|}{$g$} & $L(g;727)$\tabularnewline
\hline 
$V^{1}U^{1}V^{3}U^{1}V^{2}U^{2}$ & 87928\tabularnewline
$V^{3}U^{2}V^{1}U^{2}V^{2}U^{2}$ & 77996\tabularnewline
$V^{1}U^{1}V^{2}U^{1}V^{2}U^{3}$ & 75289\tabularnewline
$V^{2}U^{1}V^{1}U^{2}V^{1}U^{2}$ & 95183\tabularnewline
$V^{2}U^{1}V^{1}U^{1}V^{3}U^{1}$ & 42238\tabularnewline
$V^{2}U^{1}V^{1}U^{2}V^{2}U^{3}$ & 62702\tabularnewline
$V^{1}U^{1}V^{1}U^{3}V^{2}U^{1}$ & 51981\tabularnewline
$V^{1}U^{1}V^{3}U^{4}V^{1}U^{1}$ & 75716\tabularnewline
$V^{1}U^{4}V^{2}U^{1}V^{1}U^{1}$ & 79495\tabularnewline
$V^{1}U^{3}V^{2}U^{2}V^{3}U^{1}$ & 86897\tabularnewline
$V^{3}U^{1}V^{1}U^{2}V^{1}U^{3}$ & 108710\tabularnewline
$V^{2}U^{3}V^{1}U^{1}V^{3}U^{1}$ & 61549\tabularnewline
$V^{1}U^{1}V^{2}U^{4}V^{3}U^{1}$ & 87870\tabularnewline
$V^{1}U^{1}V^{2}U^{1}V^{3}U^{2}$ & 82633\tabularnewline
$V^{2}U^{4}V^{1}U^{1}V^{1}U^{1}$ & 79495\tabularnewline
$V^{4}U^{1}V^{1}U^{1}V^{1}U^{4}$ & 130737\tabularnewline
$V^{3}U^{4}V^{1}U^{1}V^{2}U^{1}$ & 72046\tabularnewline
\hline 
\end{tabular}%
\begin{tabular}{|cr|}
\hline 
$g$ ambiguous & \tabularnewline
\hline 
\multicolumn{1}{|c|}{$g$} & $L(g;727)$\tabularnewline
\hline 
$V^{1}U^{1}V^{1}U^{1}V^{1}U^{2}$ & 3193\tabularnewline
$V^{1}U^{1}V^{3}U^{1}V^{1}U^{2}$ & 2018\tabularnewline
$V^{2}U^{1}V^{2}U^{3}V^{2}U^{1}$ & 2780\tabularnewline
$V^{4}U^{1}V^{1}U^{2}V^{1}U^{1}$ & 3748\tabularnewline
$V^{1}U^{2}V^{1}U^{2}V^{3}U^{2}$ & 2780\tabularnewline
$V^{1}U^{1}V^{1}U^{1}V^{1}U^{4}$ & 2894\tabularnewline
$V^{1}U^{1}V^{1}U^{2}V^{1}U^{2}$ & 4591\tabularnewline
$V^{1}U^{3}V^{1}U^{1}V^{1}U^{1}$ & 3285\tabularnewline
$V^{1}U^{2}V^{1}U^{2}V^{1}U^{2}$ & 3331\tabularnewline
$V^{2}U^{2}V^{2}U^{1}V^{2}U^{2}$ & 3350\tabularnewline
$V^{1}U^{4}V^{1}U^{1}V^{4}U^{1}$ & 1756\tabularnewline
$V^{2}U^{1}V^{2}U^{4}V^{2}U^{1}$ & 2022\tabularnewline
$V^{2}U^{1}V^{1}U^{1}V^{2}U^{4}$ & 2937\tabularnewline
$V^{1}U^{2}V^{1}U^{1}V^{1}U^{1}$ & 3193\tabularnewline
$V^{1}U^{2}V^{2}U^{2}V^{1}U^{1}$ & 3680\tabularnewline
$V^{1}U^{2}V^{3}U^{2}V^{1}U^{2}$ & 2780\tabularnewline
$V^{1}U^{2}V^{1}U^{1}V^{4}U^{1}$ & 3748\tabularnewline
\hline 
\end{tabular}

\medskip{}

\caption{\label{tab:This-table-gives-evidence}This table gives evidence for
Conjecture \ref{conj:behaviour-of-PA}. The data is for $p=727$.
We have $|\protect\Y_{-2}(\protect\FF_{727})|=131587$. Recall $L(g;727)$
is the longest orbit of $g$ on $\protect\Y_{-2}(\protect\FF_{727})$.}
\end{table}

The issue of whether elements of $\SL_{2}(\Z)$ are conjugate to their
inverses shows up in several different areas of mathematics including
connect sum problems for manifolds \cite{CD}, the dynamics of kicked
toral automorphisms \cite{PR}, and the classification of foliations
of torus bundles over the circle \cite{GS}. This issue is explored
in depth in the article of Sarnak \cite{Sarnak} where it is related
to the theory of binary quadratic forms. A conjugacy class in $\PSL_{2}(\Z)$
is called \emph{primitive} if a representative is not a power of another
element. To each conjugacy class $[g]$ in $\PSL_{2}(\Z)$ one can
attach a number $t([g])=|\trace(g)|$. Let $\Pi$ denote the collection
of primitive hyperbolic conjugacy classes in $\PSL_{2}(\Z)$. It is
a result of Hejhal \cite{HEJHAL}, after Selberg \cite{SELBERGTRACE},
that one has the asymptotic formula
\[
\sum_{p\in\Pi,\:t(p)\leq X}1\approx\frac{X^{2}}{2\log X}.
\]
On the other hand, Sarnak shows in \cite{Sarnak} that if we write
$\Pi_{A}$ for the collection of primitive hyperbolic \emph{ambiguous}
conjugacy classes in $\PSL_{2}(\Z)$, then
\[
\sum_{p\in\Pi_{A},\:t(p)\leq X}1\approx\frac{97}{8\pi^{2}}X(\log X)^{2}.
\]
So the ambiguous classes are rare, with those having $t(p)\leq X$
taking up about a square root of the number of all primitive hyperbolic
classes with $t(p)\leq X.$

\subsection*{Acknowledgments}

We would like to thank Alex Gamburd, Alexei Entin, José Gonzalez,
Sam Payne, Doron Puder, Dhruv Ranganathan, Peter Sarnak, and Joseph
Silverman for enlightening discussions about this work. The first
version of this paper was written during the Summer Undergraduate
Research at Yale program, funded in part by Sam Payne's N.S.F. CAREER
award DMS\textendash 1149054.

\section{Proof of Theorem \ref{thm:containment}}

In this section we prove Theorem \ref{thm:containment}. Let $m_{i}$
denote the \emph{Markoff moves} on $\X_{-2}(\FF_{p})$, defined for
$i=1,2,3$ by
\[
m_{1}(x,y,z)=(yz-x,y,z),\:m_{2}(x,y,z)=(x,xz-y,z),\:m_{3}(x,y,z)=(x,y,xy-z)
\]
These moves together with permutations in $S_{3}$, permuting the
coordinates, generate all of $\PGL_{2}(\Z)\leq\Aut(\X_{\kappa})$.
We will prove Theorem \ref{thm:containment} by calculating the sign
of the $m_{i}$ and elements of $S_{3}$ as permutations of $\Y_{-2}(\FF_{p})$.

We begin by examining the subgroup $N$ of $\Aut(\X_{\kappa})$ as
it plays a special role in the action of $\Out(\F_{2})$ on $\X_{-2}^{*}(\FF_{p})$.
Recall from the Introduction the elements $n_{i}$ and the fact that
$\Out(\F_{2})$ permutes the $N$-orbits of $\X_{-2}^{*}(\FF_{p})$. 
\begin{lem}
\label{no (0,0,z)'s} There are no points in $\X_{-2}^{*}(\FF_{p})$
with zeroes in exactly two coordinate entries. Hence for $p>2$ all
orbits of $N$ in $\X_{-2}^{*}(\FF_{p})$ are of size 4.
\end{lem}
\begin{proof}
By symmetry, it suffices to check that we can have no $(0,0,z)\in\X_{-2}^{*}(\FF_{p})$,
with $z\neq0$. If $x,y=0$, substituting into (\ref{eq:Xkdef}) with
$\kappa=-2$ we obtain $0+0+z^{2}=0$ which implies $z=0$. Given
the first statement of the lemma, the second follows since no points
of $\X_{-2}^{*}(\FF_{p})$ are fixed by any $n_{i}$.
\end{proof}
Due to a result of Carlitz \cite{Carlitz}, $|\X_{-2}^{*}(\FF_{p})|=p(p+3)$
when $p\equiv1\pmod4$ and $|\X_{-2}^{*}(\FF_{p})|=p(p-3)$ when $p\equiv3\pmod4$.
Thus, 
\begin{equation}
|\Y_{-2}(\FF_{p})|=\begin{cases}
\frac{1}{4}p(p+3), & \text{if }p\equiv1\pmod{4}\\
\frac{1}{4}p(p-3), & \text{if }p\equiv3\pmod{4}
\end{cases}\label{eq: size of W*(Fp)}
\end{equation}

The following fact will be useful later.
\begin{fact}
\label{no. of consecutive residues} In $\FF_{p}$, the number of
distinct pairs of consecutive quadratic residues, both nonzero, is
exactly: 
\begin{equation}
\begin{cases}
\frac{1}{4}(p-5), & \text{when }p\equiv1\pmod4\\
\frac{1}{4}(p-3), & \text{when }p\equiv3\pmod4\,.
\end{cases}\label{no. of nonzero consecutive squares}
\end{equation}
\end{fact}
The total number of consecutive quadratic residues is found in \cite[Theorem 10-2]{Andrews}\footnote{Count the number of the solutions $(a,b)$ to $a^{2}-b^{2}=1$ in
$\FF_{p}$. To do this, count unordered pairs $\alpha:=a+b,~\beta:=a-b$
such that $\alpha\beta=1$, then discount ones that result in the
same values of $a^{2},b^{2}$.}. We discount the pair $(0,1)$ in both cases, and $(-1,0)$ when
$p\equiv1\pmod4$.
\begin{lem}
\label{lemma: no. of fixed points of m1} Let $p$ be an odd prime.
For a given $i\in\{1,2,3\}$ 
\[
\#\{(x,y,z)\in\X_{-2}^{*}(\FF_{p})\mid m_{i}(x,y,z)=(x,y,z)\}=\left\{ \begin{matrix}p-5, & p\equiv1\pmod4\\
p-3, & p\equiv3\pmod4\,.
\end{matrix}\right.
\]
\end{lem}
\begin{proof}
We will prove this formula for $m_{1}$, and it follows for $m_{2},m_{3}$
by symmetry. We have that $m_{1}(x,y,z)=(x,y,z)$ exactly when 
\begin{equation}
2x=yz.\label{eq:2x-equals-yz}
\end{equation}

Lemma \ref{no (0,0,z)'s}, equation (\ref{eq:2x-equals-yz}), and
our assumption that $(x,y,z)\neq(0,0,0)$ imply that $x,y,z\neq0$.
Substituting $x=yz/2$ into (\ref{eq:Xkdef}) we have 
\begin{equation}
y^{2}+z^{2}-\frac{y^{2}z^{2}}{4}=0.\label{eq:yz fixed when by m1}
\end{equation}
As $x$ is uniquely determined given $y,z$ by (\ref{eq:2x-equals-yz})
we count the solutions to (\ref{eq:yz fixed when by m1}) over $\FF_{p}$
.

Letting $Y=y^{2},Z=z^{2}$ we have 
\begin{equation}
Z(Y-4)=4Y\,.\label{eq: Z,Y for fixed by m1}
\end{equation}
As $y,z\neq0$ there are exactly as many $y,z$ satisfying (\ref{eq:yz fixed when by m1})
as four times the number of solutions to (\ref{eq: Z,Y for fixed by m1}).

By (\ref{eq: Z,Y for fixed by m1}), as $Y\neq0$, $Z$ is determined
uniquely by $Y$, so we just need to count possible values of $Y\neq0$
that can satisfy (\ref{eq: Z,Y for fixed by m1}). As $Y$ and $Z$
are quadratic residues, $Y-4$ must also be. Thus to count the possible
solutions to (\ref{eq: Z,Y for fixed by m1}), we just need to count
the possible values of $Y$ such that both $Y$ and $Y-4$ are nonzero
quadratic residues. This is the case if and only if $Y/4$ and $(Y-4)/4$
are consecutive nonzero quadratic residues. By (\ref{no. of nonzero consecutive squares}),
for $p\equiv1\pmod4$ (resp. $p\equiv3\pmod4$), there are $(p-5)/4$
(resp. $(p-3)/4)$ of these. This gives us our result. 
\end{proof}
\begin{lem}
\label{Sign of m_1}Suppose $p$ is an odd prime. For a given $i\in\{1,2,3\}$,
the Markoff move $m_{i}$ acts as an even permutation on $\Y_{-2}(\FF_{p})$
exactly when $p\equiv3\pmod8$. 
\end{lem}
\begin{proof}
We will show this result for $m_{1}$ and it follows by symmetry for
$m_{2},m_{3}$. Because it is an involution, the permutation induced
by $m_{1}$ on $\Y_{-2}(\FF_{p})$ is a product of 
\begin{equation}
r:=\frac{|\Y_{-2}(\FF_{p})|-|F|}{2}\label{r defn}
\end{equation}
disjoint transpositions, where $F$ is the set of fixed points of
$m_{1}$ in $\Y_{-2}(\FF_{p})$. Each of the $n_{i}$ commute with
$m_{1}$, so $\mathbf{x}\in\X_{-2}(\FF_{p})$ is fixed by $m_{1}$
if and only if all the elements of $N\cdot\mathbf{x}$ are fixed by
$m_{1}$. Consequently $|F|$ is exactly one fourth the number of
fixed points of $m_{1}$ in $\X_{-2}^{*}(\FF_{p})$ which we have
calculated in Lemma \ref{lemma: no. of fixed points of m1}. We also
recall from (\ref{eq: size of W*(Fp)}) the size of $\Y_{-2}(\FF_{p})$.
We calculate the parity of $m_{1}$ by calculating $r$ case by case:

If $p=4k+1$ 
\[
r=\frac{1}{2}\left(\frac{p^{2}+3p}{4}-\frac{p-5}{4}\right)=2(k^{2}+k)+1\equiv1\pmod2,
\]
so $m_{1}$ acts as an odd permutation. If $p=8k+7$ 
\[
r=\frac{1}{2}\left(\frac{p^{2}-3p}{4}-\frac{p-3}{4}\right)=8k^{2}+10k+3\equiv1\pmod2,
\]
so $m_{1}$ acts as an odd permutation. If $p=8k+3$ 
\[
r=\frac{1}{2}\left(\frac{p^{2}-3p}{4}-\frac{p-3}{4}\right)=8k^{2}+2k\equiv0\pmod2,
\]
so $m_{1}$ acts as an even permutation. 
\end{proof}
\begin{prop}
\label{M containment in An} The permutation group generated by the
action of $\langle m_{1},m_{2},m_{3}\rangle$ on $\Y_{-2}(\FF_{p})$
is contained in the alternating group on $\Y_{-2}(\FF_{p})$ if and
only if $p\equiv3\pmod8$. 
\end{prop}
\begin{proof}
This follows directly from Lemma \ref{Sign of m_1}. 
\end{proof}
In order to complete our proof of Theorem \ref{thm:containment},
we must check the parity of the other generators of $\PGL_{2}(\Z)$
(through which $\Out(\F_{2})$ acts). The only remaining generators
to check, aside from the Markoff moves, are those of $S_{3}$. By
Proposition \ref{M containment in An}, we know there always will
be odd permutations for $p\not\equiv3\pmod8$, so we only need to
examine the remaining case, when $p\equiv3\pmod8$.
\begin{lem}
\label{sign of (1 2)} The action of $S_{3}$ on $\Y_{-2}(\FF_{p})$
consists of even permutations when $p\equiv3\pmod{16}$. When $p\equiv11\pmod{16}$,
it consists of both even and odd permutations. 
\end{lem}
\begin{proof}
The group $S_{3}$ is generated by transpositions, and by symmetry
they all have the same parity, so it suffices to check the parity
of the action of the transposition $(1\,2)$ in the cases we consider.

Our strategy is to count the points in $\X_{-2}^{*}(\FF_{p})$ whose
$N$-orbits are fixed by $(1\,2)$. We start by counting how many
possible values $x$ can take on, then for each of those values we
will count how many points with fixed orbits there are.

The $N$-orbit of $(x,y,z)$ is fixed by $(1\,2)$ if and only if
\begin{equation}
(x,y,z)\in\{(y,x,z),(y,-x,-z),(-y,x,-z),(-y,-x,z)\},\label{eq: criterion for orbit fixed by (1 2)}
\end{equation}
which is if and only if $x=\pm y$. Note that by Lemma \ref{no (0,0,z)'s}
this rules out $x=0$.

Substituting $x=\pm y$ into (\ref{eq:Xkdef}) with $\kappa=-2$ we
reduce to two cases: 
\begin{equation}
x\neq0,\quad y=x,\quad2x^{2}+z^{2}=x^{2}z\text{, or}\label{eq:x,z in Markoff, when fixed by (1 2),1}
\end{equation}
\begin{equation}
x\neq0,\quad y=-x,\quad2x^{2}+z^{2}=-x^{2}z\text{ . }\label{eq:x,z in Markoff, when fixed by (1 2),2}
\end{equation}

For fixed $x$, in both cases we obtain quadratic equations in $z$
with discriminant $\Delta=x^{2}(x^{2}-8)$. Note that $\Delta\neq0$
as $x\neq0$ and 8 is not a quadratic residue of $\FF_{p}$ because
$p\equiv3\pmod8$ in the cases we consider. Thus (\ref{eq:x,z in Markoff, when fixed by (1 2),1})
and (\ref{eq:x,z in Markoff, when fixed by (1 2),2}) have solutions
over $\FF_{p}$ if and only if $\Delta$ is a square, which happens
if and only if $x^{2}-8$ is a square.

As we assume $p\equiv3\pmod8$, there exists\footnote{As $p\equiv3\pmod4$ we have that $\left(\frac{-1}{p}\right)=-1$
and as $p\equiv3\pmod8$ we have that $\left(\frac{2}{p}\right)=-1$.
This implies that $\left(\frac{-2}{p}\right)=\left(\frac{-8}{p}\right)=1$.} some $\alpha$ such that $\alpha^{2}=-8$. Setting $w:=x/\alpha$
we want to count how many values $w$ can take such that $x^{2}-8=-8(w^{2}+1)$
is a square, which we do by counting the number of nonzero consecutive
quadratic residues $w^{2}$ and $w^{2}+1$. From Fact \ref{no. of consecutive residues}
we have that there are $(p-3)/4$ such pairs of the form $(w^{2},w^{2}+1)$
where $w^{2}\neq0$ (as in both cases $p\equiv3\pmod4$). Each pair
of residues, $(w^{2},w^{2}+1)$, can be made by both $w$ and $-w$,
which gives us $(p-3)/2$ possible values of $w$ and hence of $x$.

For each valid $x$, those such that $\Delta$ is a square, we have
exactly four solutions total to (\ref{eq:x,z in Markoff, when fixed by (1 2),1})
and (\ref{eq:x,z in Markoff, when fixed by (1 2),2}) for $(x,y,z)$
that correspond to four points which satisfy both (\ref{eq:Xkdef})
and (\ref{eq: criterion for orbit fixed by (1 2)}) and thus four
points whose $N$-orbits are fixed by $(1\,2)$: 
\[
(x,x,z_{1}),(x,x,z_{2}),(x,-x,-z_{1}),(x,-x,-z_{2})
\]
\[
\text{ where }z_{1}=\frac{x^{2}+\sqrt{\Delta}}{2},\:z_{2}=\frac{x^{2}-\sqrt{\Delta}}{2}.
\]
Recall that as $\Delta\neq0$, we have that $z_{1}\neq z_{2}$, so
these four points are distinct. This gives us $2(p-3)$ points of
$\X_{-2}^{*}(\FF_{p})$ in total whose $N$-orbits are fixed by $(1\,2)$.
As each $N$-orbit in $\X_{-2}^{*}(\FF_{p})$ has exactly 4 points,
there are $\frac{p-3}{2}$ fixed $N$-orbits of $(1\,2)$.

To determine the parity of $(1\,2)$, we use the same method of counting
disjoint transpositions as we did for $m_{1}$ in the proof of Lemma
\ref{Sign of m_1}. Letting $F$ denote the fixed $N$-orbits of $(1\,2)$,
we examine the two cases:

If $p=16k+3$
\[
\frac{|\Y_{-2}(\FF_{p})|-|F|}{2}=\frac{1}{2}\left(\frac{p(p-3)}{4}-\frac{p-3}{2}\right)=2k(16k+1)\equiv0\pmod2,
\]
so $(1\,2)$ acts as an even permutation. 

If $p=16k+11$ 
\[
\frac{|\Y_{-2}(\FF_{p})|-|F|}{2}=\frac{1}{2}\left(\frac{p(p-3)}{4}-\frac{p-3}{2}\right)=32k^{2}+34k+9\equiv1\pmod2,
\]
so $(1\,2)$ acts as an odd permutation. The lemma follows directly
from this result. 
\end{proof}
Theorem \ref{thm:containment} now follows directly from Lemma \ref{sign of (1 2)}
and Proposition \ref{M containment in An}. 

\section{Lower bound on the longest orbit}

\subsection{Background on the free group.\label{subsec:Background-on-the-free-group}}

Here we give necessary background about the free group $\F_{2}$ and
its automorphisms. We write $X$ and $Y$ for the generators of $\F_{2}$.
Firstly, we always assume words in $\F_{2}$ are reduced, meaning
positive powers of $X$ do not appear beside negative powers, and
similarly for $Y$. Following \cite{PP} we make the following definition.
\begin{defn}
A word $w\in\F_{2}$ is \emph{monotone }if for each letter $X$ or
$Y$, all the exponents of this letter in $w$ have the same sign.
\end{defn}
We need the following proposition that appears in Parzanchevski and
Puder \cite[Prop. 3.5]{PP}.
\begin{prop}
\label{prop:monotone reps}Any element $\Phi$ of $\Out(\F_{2})$
has a representative in $\Aut(\F_{2})$ of the form
\begin{equation}
\hat{\Phi}:(X,Y)\mapsto(w_{1},w_{2})\label{eq:phi+exp}
\end{equation}
where $w_{1}$ and $w_{2}$ are monotone words in $\F_{2}$. 
\end{prop}
In the setting of Proposition \ref{prop:monotone reps} we say that
$\hat{\Phi}$ is \emph{monotone.} Suppose $\hat{\Phi}\in\Aut(\F_{2})$
as in (\ref{eq:phi+exp}) is monotone, with 
\begin{align*}
w_{i} & =X^{\alpha_{1}^{i}}Y^{\beta_{1}^{i}}X^{\alpha_{2}^{i}}Y^{\beta_{2}^{i}}\ldots X^{\alpha_{t_{i}}^{i}}Y^{\beta_{t_{i}}^{i}}
\end{align*}
for some $\alpha_{j}^{i},\beta_{j}^{i},t_{i}\in\Z$. We identify $\Z^{2}\cong\F_{2}/[\F_{2},\F_{2}]$
by the basis induced by $X,Y$. Then $\hat{\Phi}$ acts on $\Z^{2}$
by the matrix 
\[
\left(\begin{array}{cc}
a_{1} & a_{2}\\
b_{1} & b_{2}
\end{array}\right)\in\GL_{2}(\Z)
\]
where $a_{i}=\sum_{j}\alpha_{j}^{i}$ and $b_{i}=\sum_{j}\beta_{j}^{i}$.
Moreover the values $a_{i}$ and $b_{i}$ are uniquely determined
by $\Phi\in\Out(\F_{2})$ and vice versa. We pass freely between these
representations of $\Phi$ in the rest of the paper.

\subsection{Algebraic setup }

We consider the affine scheme over $\Z[\kappa]$ 
\[
\X:=\Spec(R)
\]
where
\[
R:=\Z[\kappa,x,y,z]/I,\quad I:=(x^{2}+y^{2}+z^{2}-xyz-2-\kappa).
\]
For particular choice of $\kappa\in\Z$ we obtain a scheme over $\Z$
that we denote by 
\[
\X_{\kappa}:=\Spec(R_{\kappa}),
\]
 where
\[
R_{\kappa}:=\Z[x,y,z]/I_{\kappa},\quad I_{\kappa}:=(x^{2}+y^{2}+z^{2}-xyz-\kappa).
\]
In the case of $\kappa=-2$ one obtains the Markoff surface. The group
$\Out(\F_{2})\cong\GL_{2}(\Z)$ acts on $\X$ by automorphisms of
schemes over $\Z[\kappa]$ and for each $\kappa$, $\GL_{2}(\Z)$
acts on $\X_{\kappa}$ by automorphisms.

\subsection{The Cayley Cubic }

When $\kappa=2$, $\X_{2}$ is \emph{Cayley's cubic surface} \cite{Cayley}.\emph{
}In fact $\X_{2}$ is closely related to the split torus $\G_{m}^{2}$;
we heavily exploit this fact in the sequel. To see this, let $\tilde{\X}_{2}:=\Spec(\tilde{R}_{2})$
where
\[
\tilde{R}_{2}:=\Z[x,y,z,\delta,\eta]/J_{2},\quad J_{2}:=(x^{2}+y^{2}+z^{2}-xyz-4,\delta^{2}-x\delta+1,\eta^{2}-y\eta+1).
\]
The mapping 
\begin{eqnarray}
\tilde{R}_{2} & \to & \O_{\G_{m}^{2}}:=\Z[\delta,\delta^{*},\eta,\eta^{*}]/(\delta\delta^{*}-1,\eta\eta^{*}-1)\nonumber \\
x & \mapsto & \delta+\delta^{*}\label{eq:embed1}\\
y & \mapsto & \eta+\eta^{*}\label{eq:embed2}\\
z & \mapsto & \delta\eta+\delta^{*}\eta^{*}\label{eq:embed3}
\end{eqnarray}
and $\delta,\eta\mapsto\delta,\eta$ induces an isomorphism $\tilde{\X}_{2}\cong\G_{m}^{2}$.
The inclusion of $R_{2}\to\tilde{R}_{2}$ induces a map
\[
\G_{m}^{2}\cong\tilde{\X}_{2}\to\X_{2}.
\]
There is an action of $\GL_{2}(\Z)$ on $\G_{m}^{2}$ by
\begin{align}
g(\delta)=\delta^{a}\eta^{c}, & \quad g(\eta)=\delta^{b}\eta^{d},\nonumber \\
g(\delta^{*})=(\delta^{*})^{a}(\eta^{*})^{c}, & \quad g(\eta^{*})=(\delta^{*})^{b}(\eta^{*})^{d},\label{eq:GL_2-action}
\end{align}
for $g=\left(\begin{array}{cc}
a & b\\
c & d
\end{array}\right)\in\GL_{2}(\Z)$. We interpret $\delta^{-n}=(\delta^{*})^{n}$ for $n\in\Z$ and similarly
$\eta^{-n}=(\eta^{*})^{n}$. 

Let $\iota$ be the map $\iota:R_{2}\to\O_{\G_{m}^{2}}$ defined by
the inclusion $R_{2}\to\tilde{R}_{2}$ followed by the map $\tilde{R}_{2}\to\O_{\G_{m}^{2}}$
given by (\ref{eq:embed1}), (\ref{eq:embed2}), (\ref{eq:embed3}).
This induces a map
\[
\iota^{*}:\G_{m}^{2}\to\X_{2}.
\]

\begin{lem}
The map $\iota^{*}:\G_{m}^{2}\to\X_{2}$ is $\GL_{2}(\Z)$-equivariant. 
\end{lem}
\begin{proof}
Recall $U,V$ from our Introduction. The lemma can be checked by noting
that $\GL_{2}(\Z)$ is generated by $U$, $V$ and $(1\,2)$, and
these act on $R_{2}$ by
\[
U(x,y,z)=(z,y,zy-x),\quad V(x,y,z)=(x,z,xz-y),\quad(1\,2)(x,y,z)=(y,x,z).
\]
Then taking $V$ as an example, $V=\left(\begin{array}{cc}
1 & 1\\
0 & 1
\end{array}\right)$, and therefore using \eqref{eq:embed1}, \eqref{eq:embed2}, \eqref{eq:embed3}
and (\ref{eq:GL_2-action}) gives
\begin{align*}
V\circ\iota(x,y,z) & =V(\delta+\delta^{*},\eta+\eta^{*},\delta\eta+\delta^{*}\eta^{*})\\
 & =(\delta+\delta^{*},\delta\eta+\delta^{*}\eta^{*},\delta^{2}\eta+(\delta^{*})^{2}\eta^{*})=\iota(x,z,xz-y)=\iota V(x,y,z).
\end{align*}
The calculations for $U$ and $(1\,2)$ are similar. 
\end{proof}
Our current goal is to calculate the action of a given $g\in\GL_{2}(\Z)$
on $\X_{2}$. We will do this by exploiting the embedding $\iota:R_{2}\to\O_{\G_{m}^{2}}$.
It is convenient for our analysis to exclude certain edge cases, so
we make the following definition.
\begin{defn}[Good matrices]
\label{def:good}Let $g=\left(\begin{array}{cc}
a & b\\
c & d
\end{array}\right)\in\GL_{2}(\Z)$. We say $g$ is \emph{good }if $a,b,c,d\geq2$.
\end{defn}
We use the notation $O_{x}(x^{n})$ for the class of polynomials containing
terms with $x$-degree $\leq n$, that is, with no monomial summand
containing a power of $x$ greater than $n$.
\begin{prop}
\label{prop:Cayley-polys}For each coprime $a,c\in\Z$ with $a\geq2$,
$c\geq2$ there are polynomials $p_{a,c}$, $q_{a,c}\in\Z[x,y]$ with
the following properties. Assume $g=\left(\begin{array}{cc}
a & b\\
c & d
\end{array}\right)\in\GL_{2}(\Z)$ and that $g$ is good. 
\begin{enumerate}
\item We have 
\[
g(x)=p_{a,c}+q_{a,c}z\:\bmod I_{2},\quad g(y)=p_{b,d}+q_{b,d}z\:\bmod I_{2}
\]
where $x,y\in R_{2}$ are the first two coordinate functions on $\X_{2}$.
Here when we make statements that relate elements of $\Z[x,y]$ to
elements of $R_{\kappa}$ we always use the natural inclusion $\Z[x,y]\to R_{\kappa}$.
\item We have
\[
D:=\det\left(\begin{array}{cc}
p_{a,c}-x & q_{a,c}\\
p_{b,d}-y & q_{b,d}
\end{array}\right)=x^{a+b-1}D_{0}+O_{x}(x^{a+b-2})
\]
with $D_{0}\in\Z[y]$, with $D_{0}\neq0$ and monic, up to a sign,
and $\deg(D_{0})=|d-c|-1\geq0$.
\end{enumerate}

\end{prop}

\begin{proof}
Working in $\tilde{R}_{2}$, write
\[
(x,y,z)=(\delta+\delta^{*},\eta+\eta^{*},\delta\eta+\delta^{*}\eta^{*}).
\]
It will be useful to use the notations\footnote{$c(\delta)$ should be thought of as $2\cos(\theta)$ for abstract
$\theta$ such that $\delta=\exp(i\theta)$.} $c(\delta^{n}\eta^{m}):=\delta^{n}\eta^{m}+(\delta^{*})^{n}(\eta^{*})^{m},$
and $s(\delta^{n}\eta^{m}):=\delta^{n}\eta^{m}-(\delta^{*})^{n}(\eta^{*})^{m}$,
interpreting $\delta^{-1}$ as $\delta^{*}$ as before to extend the
definitions of $c$ and $s$ to include negative powers of $\delta$
and $\eta$. Note that analogs of trigonometric formulas hold also
for these functions. 

Now, $g(x,y)$ is given by the expression
\begin{eqnarray}
(g(x),g(y)) & = & (c(\delta^{a}\eta^{c}),c(\delta^{b}\eta^{d}))\nonumber \\
 & = & \frac{1}{2}(c(\delta^{a})c(\eta^{c})+s(\delta^{a})s(\eta^{c}),c(\delta^{b})c(\eta^{d})+s(\delta^{b})s(\eta^{d})).\label{eq:groupaction1}
\end{eqnarray}
We have 
\begin{align}
c(\delta^{a}) & =2T_{a}\left(\frac{x}{2}\right),\quad c(\eta^{c})=2T_{c}\left(\frac{y}{2}\right),\label{eq:cos-Chebyshev}
\end{align}
 where for $a\geq0$, $T_{a}\in\Z[t]$ is the Chebyshev polynomial
of the first kind. Although we work in $\Z[\frac{1}{2}]\otimes\tilde{R}_{2}$
throughout the proof, our final results will hold in $R_{2}$. Similarly
for $a,c\geq2$ 
\begin{eqnarray}
s(\delta^{a}) & = & U_{a-1}\left(\frac{x}{2}\right)s(\delta),\quad s(\eta^{c})=U_{c-1}\left(\frac{y}{2}\right)s(\eta),\label{eq:sine-chebyshev}
\end{eqnarray}
 where $U_{a}\in\Z[t]$ is the Chebyshev polynomial of the second
kind. 

Using this we obtain from (\ref{eq:groupaction1}) and (\ref{eq:cos-Chebyshev}),
(\ref{eq:sine-chebyshev}) the expression
\[
(g(x),g(y))=(P_{a,c}(x,y,z),P_{b,d}(x,y,z)),
\]
where
\[
P_{a,c}(x,y,z):=2T_{a}\left(\frac{x}{2}\right)T_{c}\left(\frac{y}{2}\right)+\frac{1}{2}U_{a-1}\left(\frac{x}{2}\right)U_{c-1}\left(\frac{y}{2}\right)(2z-xy).
\]
To obtain this expression, we used that $s(\delta)s(\eta)=2z-xy$.
The key point is that $P_{a,c}(x,y,z)$ is linear in $z,$ and we
obtain Part 1 of the proposition with 
\begin{eqnarray*}
p_{a,c}(x,y) & := & 2T_{a}\left(\frac{x}{2}\right)T_{c}\left(\frac{y}{2}\right)-\frac{1}{2}xyU_{a-1}\left(\frac{x}{2}\right)U_{c-1}\left(\frac{y}{2}\right),\\
q_{a,c}(x,y) & := & U_{a-1}\left(\frac{x}{2}\right)U_{c-1}\left(\frac{y}{2}\right).
\end{eqnarray*}
Using that $2T_{a}\left(\frac{t}{2}\right)$ and $U_{a-1}\left(\frac{t}{2}\right)$
are monic in $t$ for $a\geq1$ of degrees $a$ and $a-1$ respectively,
we get that the leading $x$-degree contribution to $p_{a,c}$ is
$x^{a}u_{c}$ where
\begin{align}
u_{c}(y): & =T_{c}\left(\frac{y}{2}\right)-\frac{y}{2}U_{c-1}\left(\frac{y}{2}\right)=-U_{c-2}\left(\frac{y}{2}\right).\label{eq:good-expression}
\end{align}
The last equality uses the sum of angle formula for sine together
with the connection between Chebyshev polynomials and trigonometric
functions. The leading $x$-degree contribution to $q_{a,c}$ is more
easily seen to be $x^{a-1}v_{c}$ where 
\[
v_{c}(y):=U_{c-1}\left(\frac{y}{2}\right).
\]
This concludes our calculations for the pair $a,c$. Since $g$ is
good, we have $b,d\geq2$ and so the calculation of $P_{b,d}$ and
$p_{b,d}$, $q_{b,d}$ is analogous to the preceding one, replacing
$a,c\mapsto b,d$. 

\textbf{Calculation of $D$ and $D_{0}$. }Note that since $g$ is
good, we must have $c\neq d$. Indeed, if $c=d$ then from the determinant
of $g$ being $\pm1$, one sees that $c$ and $d$ are coprime, which
cannot happen since $c=d\geq2$. Since $a\geq2$, the $-x$ term in
the determinant does not contribute to the largest $x$-degree term.
We get
\begin{align*}
D & =x^{a+b-1}\left(u_{c}v_{d}-v_{c}u_{d}\right)\\
 & =x^{a+b-1}\left(-U_{c-2}\left(\frac{y}{2}\right)U_{d-1}\left(\frac{y}{2}\right)+U_{d-2}\left(\frac{y}{2}\right)U_{c-1}\left(\frac{y}{2}\right)\right)+O_{x}(x^{a+b-2})\\
 & =x^{a+b-1}\sign(d-c)U_{|d-c|-1}\left(\frac{y}{2}\right)+O_{x}(x^{a+b-2}),
\end{align*}
where one can use sum of angle formulas for sine to get the final
identity.
\end{proof}

\subsection{The deformation from $\kappa=2$.}

It is well known since work of Fricke \cite{FRICKEKLEIN} that to
each $w\in\F_{2}$ the induced \emph{word map $w:\SL_{2}(\C)\times\SL_{2}(\C)\to\SL_{2}(\C)$
}has 
\[
\tr(w(A,B))=P_{w}(x,y,z)
\]
for unique $P_{w}\in\Z[x,y,z]$, where $x=\tr(A),y=\tr(B),z=\tr(AB)$.
Indeed this follows from repeated applications of the identity
\begin{equation}
\tr(uv)=\tr(u)\tr(v)-\tr(u^{-1}v),\quad u,v\in\SL_{2}(\C).\label{eq:UV-mutation}
\end{equation}
If $\theta\in\Aut(\F_{2})$ acts by $\theta(X,Y)=(w_{1}(X,Y),w_{2}(X,Y))$
then $\theta$ acts on the coordinate functions $x,y\in R$ by
\[
\theta(x)=P_{w_{1}}(x,y,z),\quad\theta(y)=P_{w_{2}}(x,y,z),\quad P_{w_{i}}\in\Z[x,y,z].
\]

Define the $(x,z)$-degree of a monomial $x^{\alpha}y^{\beta}z^{\gamma}\kappa^{\delta}$
to be $\alpha+\gamma$, and define the $(x,z)$-degree of a polynomial
$f$ in $\Z[\kappa,x,y,z]$ to be the maximum of the $(x,z)$-degrees
of the monomials with nonzero coefficients in $f$. We write $f^{(N)}$
for the $(x,z)$-degree $N$ piece of $f$, that is, the part comprised
of monomials of $(x,z)$-degree $N$.
\begin{lem}
\label{lem:xz-degree-word-map}Write $X,Y$ for fixed generators of
$\F_{2}$. Let 
\[
w=X^{\alpha_{1}}Y^{\beta_{1}}X^{\alpha_{2}}Y^{\beta_{2}}\ldots X^{\alpha_{t}}Y^{\beta_{t}}
\]
 be a monotone word, with every $\alpha_{i},\beta_{i}\neq0$. Let
$a=\sum_{i=1}^{t}\alpha_{i}$ and $b=\sum_{i=1}^{t}\beta_{i}$. The
$(x,z)$-degree of $P_{w}$ is $\leq|a|$.
\end{lem}
\begin{proof}
Assume for ease of exposition that all $\alpha_{i},\beta_{i}$ are
positive, so $a,b>0$. This will be the case for words arising from
good elements of $\GL_{2}(\Z)$. The proof is by induction on the
partial order $\preceq$ defined by the following moves:
\begin{itemize}
\item If any $\alpha_{i}$ has $\alpha_{i}\geq2$ then $w',w''\preceq w$
for either $w',w''$ obtained by replacing $\alpha_{i}\mapsto\alpha_{i}-1$
or $\alpha_{i}\mapsto\alpha_{i}-2$. Then (\ref{eq:UV-mutation})
yields
\[
P_{w}(x,y,z)=xP_{w'}(x,y,z)-P_{w''}(x,y,z).
\]
Note if the lemma holds for $P_{w'}$ and $P_{w''}$, it holds for
$P_{w}$.
\item If any $\beta_{i}$ has $\beta_{i}\geq2$ then we perform the replacements
$\beta_{i}\mapsto\beta_{i}-1$ or $\beta_{i}\mapsto\beta_{i}-2$ to
form $w',w''$ and declare $w',w''\preceq w$. By the same logic as
before, $P_{w}(x,y,z)=yP_{w'}(x,y,z)-P_{w''}(x,y,z)$ so if the lemma
holds for $w'$ and $w''$ it holds for $w$.
\item We identify all words with their cyclically reduced conjugates. This
doesn't change $P_{w}$.
\end{itemize}
To put this all together, note that any minimal cyclically reduced
word with respect to $\preceq$ has all the $\alpha_{i}=\beta_{i}=1$.
If all the $\alpha_{i}$ and $\beta_{i}$ are $1$, and $w$ is cyclically
reduced, then $w$ is a power of $XY$ or $YX$ and e.g. if $w=(XY)^{n}$
then $a=n$. On the other hand, $P_{(XY)^{n}}(x,y,z)=2T_{n}\left(\frac{z}{2}\right)$
has $(x,z)$-degree $n$ as required (this also shows the statement
of the lemma is sharp).
\end{proof}
Our next goal is to show, in the present context, that $P_{w}$ are
equal in $R$ to functions that are linear in $z$ and such that certain
terms have no dependence on $\kappa$.
\begin{lem}
\label{lem:Degree-control}If $\theta\in\Aut(\F_{2})$ satisfies $\theta(X,Y)=(w_{1}(X,Y),w_{2}(X,Y))$,
then 
\[
P_{w_{1}}(x,y,z)=U_{w_{1}}+V_{w_{1}}z\:\bmod I,\quad P_{w_{2}}(x,y,z)=U_{w_{2}}+V_{w_{2}}z\:\bmod I
\]
where $U_{w_{i}},V_{w_{i}}\in\Z[\kappa,x,y]$ have the following property.
If $N_{i}$ is at least the $(x,z)$-degree of $P_{w_{i}}$ then 
\begin{enumerate}
\item $U_{w_{i}}=x^{N_{i}}U_{w_{i}}^{0}+O_{x}(x^{N_{i}-1})$ with $U_{w_{i}}^{0}\in\Z[y]$.
\item $V_{w_{i}}=x^{N_{i}-1}V_{w_{i}}^{0}+O_{x}(x^{N_{i}-2})$ with $V_{w_{i}}^{0}\in\Z[y]$. 
\end{enumerate}
In particular, $U_{w_{i}}^{0}$ and $V_{w_{i}}^{0}$ do not depend
on $\kappa$.
\end{lem}
\begin{proof}
Transform $P_{w_{1}}(x,y,z)$ by replacing each monomial of the form
$x^{\alpha}y^{\beta}z^{\gamma}$ with $\gamma\geq2$ by 
\begin{equation}
x^{\alpha}y^{\beta}z^{\gamma}\mapsto x^{\alpha}y^{\beta}z^{\gamma-2}(xyz-x^{2}-y^{2}+2+\kappa),\label{eq:replacement1}
\end{equation}
these two terms are equal $\bmod I$. Moreover this replacement has
the following properties: if $p,q\in\Z[\kappa,x,y,z]$ and $p\mapsto q$
in this manner then
\begin{itemize}
\item The $(x,z)$-degree of $q$ is at most the $(x,z)$-degree of $p$.
\item Let $N_{1}$ be at least the $(x,z)$-degree of $p$ and let $p^{(N_{1})}$
be the $(x,z)$-degree $N_{1}$ component of $p$ and similarly define
$q^{(N_{1})}$. If $N_{1}$ is larger than the $(x,z)$-degree of
$p$ then $p^{(N_{1})}$ is zero. If $p^{(N_{1})}\in\Z[x,y,z]$ then
$q^{(N_{1})}\in\Z[x,y,z]$ (so doesn't depend on $\kappa$). This
follows since $q^{(N_{1})}$ is obtained from $p^{(N_{1})}$ by replacement
of all monomials of the form $x^{\alpha}y^{\beta}z^{\gamma}$ with
$\gamma\geq2$ by
\[
x^{\alpha}y^{\beta}z^{\gamma}\mapsto x^{\alpha}y^{\beta}z^{\gamma-2}(xyz-x^{2})=x^{\alpha+1}y^{\beta}z^{\gamma-2}(yz-x).
\]
Monomials $x^{\alpha}y^{\beta}z^{\gamma}$ with $\gamma\leq1$ are
left unaltered. 
\end{itemize}
The effect of iterating this reduction, beginning with the fact that
$P_{w_{1}}\in\Z[x,y,z]$, yields polynomials $U_{w_{1}},V_{w_{1}}\in\Z[\kappa,x,y]$
such that $P_{w_{1}}=U_{w_{1}}+V_{w_{1}}z\bmod I$, the $(x,z)$-degree
of $U_{w_{1}}+V_{w_{1}}z$ is $\leq N_{1}$, and $(U_{w_{1}}+V_{w_{1}}z)^{(N_{1})}\in\Z[x,y,z]$.
This means that the $x$-degree of $U_{w_{1}}$ is $\leq N_{1}$ and
$U_{w_{1}}^{(N_{1})}\in\Z[x,y]$. Similarly the $x$-degree of $V_{w_{1}}$
is $\leq N_{1}-1$ and $V_{w_{1}}^{(N_{1}-1)}\in\Z[x,y].$ Performing
this reduction also for $P_{w_{2}}$ with $N_{2}$ in place of $N_{1}$
establishes the result.
\end{proof}
\begin{prop}
\label{prop:fixed-point-determinant}For each coprime $a,c\in\Z$
with $a\geq2$, $c\geq2$ there are polynomials $\tilde{p}_{a,c}$,
$\tilde{q}_{a,c}\in\Z[x,y]$ with the following properties. Assume
$g=\left(\begin{array}{cc}
a & b\\
c & d
\end{array}\right)\in\GL_{2}(\Z)$ and that $g$ is good. 
\begin{enumerate}
\item We have
\[
g(x)=\tilde{p}_{a,c}+\tilde{q}_{a,c}z\:\bmod I,\quad g(y)=\tilde{p}_{b,d}+\tilde{q}_{b,d}z\:\bmod I.
\]
\item Moreover, 
\[
\tilde{D}:=\det\left(\begin{array}{cc}
\tilde{p}_{a,c}-x & \tilde{q}_{a,c}\\
\tilde{p}_{b,d}-y & \tilde{q}_{b,d}
\end{array}\right)\in\Z[\kappa,x,y]
\]
is given by
\[
\tilde{D}=x^{a+b-1}D_{0}+O_{x}(x^{a+b-2})
\]
where $D_{0}\in\Z[y]$ is the same quantity as in Proposition \ref{prop:Cayley-polys},
for the same $g$.
\end{enumerate}
\end{prop}
\begin{proof}
Let $\hat{\Phi}\in\Aut(\F_{2})$ be a monotone automorphism representing
$g$, given by Proposition \ref{prop:monotone reps}. We consider
$\Phi(x)=g(x)$, the calculation of $\Phi(y)$ is similar. Let $w_{1}$
and $w_{2}$ be the monotone words appearing in the expression (\ref{eq:phi+exp})
for $\hat{\Phi}$. We have $\Phi(x)=P_{w_{1}}(x,y,z)$. This has $(x,z)$-degree
$\leq a$ by Lemma \ref{lem:xz-degree-word-map}. Note that since
we know $w_{1}$ is monotone, we can conjugate $w_{1}$ to be of the
form as in Lemma \ref{lem:xz-degree-word-map} without changing $a$
or $P_{w_{1}}(x,y,z)$.

Applying Lemma \ref{lem:Degree-control} with $N_{1}=a$ we can write
\[
\Phi(x)=U_{w_{1}}^{0}x^{a}+U'_{w_{1}}+(V_{w_{1}}^{0}x^{a-1}+V'_{w_{1}})z,
\]
where $U_{w_{1}}^{0},V_{w_{1}}^{0}\in\Z[y]$, $U'_{w_{1}}\in\Z[\kappa,x,y]$
has $x$-degree $\leq a-1$ and $V'_{w_{1}}\in\Z[\kappa,x,y]$ has
$x$-degree $\leq a-2$. We obtain the first part of the proposition
with
\begin{equation}
\tilde{p}_{a,c}:=U_{w_{1}}^{0}x^{a}+U'_{w_{1}},\quad\tilde{q}_{a,c}:=V_{w_{1}}^{0}x^{a-1}+V'_{w_{1}}.\label{eq:tilde_defs}
\end{equation}
Similarly $\tilde{p}_{c,d}$ and $\tilde{q}_{c,d}$ are obtained by
replacing $w_{1}$ by $w_{2}$ and $a,c\mapsto b,d$. Note at this
moment we do not know that $U_{w_{1}}^{0}$ and $U_{w_{2}}^{0}$ are
non-zero.

Let $\pi$ be the evaluation map $\Z[\kappa,x,y,z]\to\Z[x,y,z]$ sending
$\kappa\mapsto2$. We must have in $R_{2}$
\begin{equation}
\pi\left(\tilde{p}_{a,c}+\tilde{q}_{a,c}z\right)\equiv\pi(\tilde{p}_{a,c})+\pi(\tilde{q}_{a,c})z\equiv p_{a,c}+q_{a,c}z\bmod I_{2}\label{eq:equialtiy-of-expressions}
\end{equation}
where $p_{a,b}$ and $q_{a,b}$ are the polynomials from Proposition
\ref{prop:Cayley-polys}. This is because they both describe how $g$
maps the coordinate function $x$. In the other hand, since the left
and right hand sides of (\ref{eq:equialtiy-of-expressions}) differ
by a function that is linear in $z$, this difference must be zero
since $0$ is the only element of $I_{2}$ that is linear in $z$.
So the identity (\ref{eq:equialtiy-of-expressions}) actually holds
in $\Z[x,y,z]$. This means $\pi(\tilde{p}_{a,c})=p_{a,c}$, $\pi(\tilde{q}_{a,c})=\pi(q_{a,c})$
and the same replacing $a,c\mapsto b,d$. 

This implies, if $D$ is the quantity obtained in Proposition \ref{prop:Cayley-polys},
that
\[
\pi(\tilde{D})=D.
\]
From (\ref{eq:tilde_defs}) we have
\[
\tilde{D}=x^{a+b-1}(U_{w_{1}}^{0}V_{w_{2}}^{0}-U_{w_{2}}^{0}V_{w_{1}}^{0})+O_{x}(x^{a+b-2}).
\]
Since the $x^{a+b-1}$ coefficient of $\tilde{D}$ doesn't depend
on $\kappa$, and equals $D_{0}$ when evaluated at $\kappa=2$, it
must be equal to $D_{0}$. This completes the proof.
\end{proof}

\subsection{Proof of Theorem \ref{thm:long-cycles-uniform}.}

Fix $\kappa\in\Z$. Our proof relies on proving that reasonably small
powers of $g$ have few fixed points. The following lemma combines
our previous estimates with a variant of the Schwartz-Zippel Lemma
\cite{Schwartz,Zippel}.
\begin{lem}
\label{lem:few-fixed-points}Let $g=\left(\begin{array}{cc}
a & b\\
c & d
\end{array}\right)\in\GL_{2}(\Z)$ with $|a|,|b|,|c|,|d|\geq2$, then for any $\kappa\in\Z$, $g$ has
fewer than $2p(||d|-|c||+|a|+|b|)$ fixed points in $\X_{\kappa}(\FF_{p}).$
\end{lem}
\begin{proof}
First if $ab<0$ then conjugating by $\left(\begin{array}{cc}
1 & 0\\
0 & -1
\end{array}\right)$ gives a new matrix which has $ab>0$. Now if $a<0$, multiplying
by $-I$ gives a new matrix with $a,b,c,d\geq2$, i.e. the resulting
matrix is good. These operations do not change the conjugacy class
of the matrix in $\PGL_{2}(\Z)$, therefore the number of fixed points
on $\X_{\kappa}(\FF_{p})$, and neither do they change the quantity
$2p(||d|-|c||+|a|+|b|)$. So we may assume without loss of generality
that $g$ is good.

In this proof we distinguish a specific fixed value $\kappa_{0}$
from the generic parameter $\kappa$ of $\Z[\kappa,x,y]$. Let $\tilde{p}_{a,c},\tilde{q}_{a,c}$
be the polynomials from Proposition \ref{prop:fixed-point-determinant},
and let $p_{a,c}^{\kappa_{0}}$, $q_{a,c}^{\kappa_{0}}$ be the images
of $\tilde{p}_{a,c}$, $\tilde{q}_{a,c}$ under the evaluation map
\[
\pi_{\kappa_{0}}:\Z[\kappa,x,y]\to\Z[x,y],\quad\kappa\mapsto\kappa_{0}.
\]
If $(X,Y,Z)\in\X_{\kappa_{0}}(\FF_{p})$ is a fixed point of $g$,
then from Proposition \ref{prop:fixed-point-determinant} we know
$p_{a,c}^{\kappa_{0}}(X,Y)+q_{a,c}^{\kappa_{0}}(X,Y)Z=X$ and $p_{b,d}^{\kappa_{0}}(X,Y)+q_{b,d}^{\kappa_{0}}(X,Y)Z=Y$
so
\begin{equation}
\left(\begin{array}{cc}
p_{a,c}^{\kappa_{0}}(X,Y)-X & q_{a,c}^{\kappa_{0}}(X,Y)\\
p_{b,d}^{\kappa_{0}}(X,Y)-Y & q_{b,d}^{\kappa_{0}}(X,Y)
\end{array}\right)\left(\begin{array}{c}
1\\
Z
\end{array}\right)\equiv\left(\begin{array}{c}
0\\
0
\end{array}\right)\bmod p.\label{eq:det}
\end{equation}
In particular, the determinant $D^{\kappa_{0}}\in\Z[x,y]$ of this
matrix must be zero when evaluated at $(X,Y)\in\FF_{p}^{2}$. But,
recalling Proposition \ref{prop:fixed-point-determinant} and its
notation,
\[
D^{\kappa_{0}}=\pi_{\kappa_{0}}(D)=\pi_{\kappa_{0}}\left(x^{a+b-1}D_{0}+O_{x}(x^{a+b-2})\right)=x^{a+b-1}D_{0}+O_{x}(x^{a+b-2})
\]
by using that $D_{0}\in\Z[y]$. 

Proposition \ref{prop:fixed-point-determinant} tells us that for
$Y\in\FF_{p}$ with $D_{0}(Y)\neq0$, the polynomial in $\FF_{p}[x]$
obtained by evaluating $D^{\kappa_{0}}$ at $y=Y$ has degree $a+b+1$.
So recalling from Proposition \ref{prop:Cayley-polys} that $D_{0}$
is monic up to a sign with degree $|d-c|-1$, there are at most
\[
p.(||d|-|c||-1)+p(a+b+1)=p(||d|-|b||+|a|+|c|)
\]
pairs $(X,Y)\in\FF_{p}^{2}$ for which (\ref{eq:det}) can hold. On
the other hand, since $X^{2}+Y^{2}+Z^{2}=XYZ+2+\kappa_{0}$, given
$X,Y$ for which (\ref{eq:det}) holds, there are at most two possible
$Z$ with $(X,Y,Z)\in\X_{\kappa_{0}}(\FF_{p})$.
\end{proof}
\begin{proof}[Proof of Theorem \ref{thm:long-cycles-uniform}]

Given hyperbolic $g$ in $\GL_{2}(\Z)$, we consider powers $g^{n}$
of this element. Let $\lambda$ be the eigenvalue of $g$ of largest
modulus. Diagonalizing $g$ we have

\[
g^{n}=\left(\begin{array}{cc}
Q_{11}(\lambda^{n},\lambda^{-n}) & Q_{12}(\lambda^{n},\lambda^{-n})\\
Q_{21}(\lambda^{n},\lambda^{-n}) & Q_{22}(\lambda^{n},\lambda^{-n})
\end{array}\right)
\]
where the $Q_{ij}$ are quadratic forms depending on $g$. It is possible
to check that since $g$ is hyperbolic, all the coefficients of $g^{n}$
are unbounded as $n\to\infty$ in the sense that for all $M>0$, there
is $N(M)$ such that when $n>N(M)$, $|(g^{n})_{ij}|>M$ for all $1\leq i,j\leq2$.

Indeed, if $g$ is hyperbolic it cannot fix $[1;0]$ or $[0;1]$ in
the action of $\GL_{2}(\Z)$ on $P^{1}(\R)$. So $g$ has an attracting
fixed point $z_{+}$ in $P^{1}(\R)$ that is distinct from than $[1;0]$
and $[0;1]$. Of course the same is true for the transpose $g^{T}$.
This means, projectively, $g^{n}$ converges to a matrix with all
entries nonzero. Since $\GL_{2}(\Z)$ is discrete, at least one entry
of $g^{n}$ is unbounded, hence all the entries are.

Note this implies that for $n\geq n_{0}(g)$, $g^{n}$ satisfies the
hypothesis of Lemma \ref{lem:few-fixed-points}. Noting that there
is $C=C(g)$ such that all coefficients of $g^{n}$ are $\leq C\lambda^{n}$,
Lemma \ref{lem:few-fixed-points} gives that $g^{n}$ has fewer than
$8Cp|\lambda|^{n}$ fixed points on $\X_{\kappa}(\FF_{p})$ when $n\geq n_{0}(g)$.

We also need a bound on the number of fixed points of $g^{n}$ when
$n<n_{0}(g)$. In this case, we have that $g^{n\lceil n_{0}(g)/n\rceil}$
satisfies the hypothesis of Lemma \ref{lem:few-fixed-points}, so
it has fewer than $Mp$ fixed points, where
\[
M=8\max\{\:|(g^{n\lceil n_{0}(g)/n\rceil})_{i,j}|\::\:1\leq i,j\leq2,\:,\:1\leq n<n_{0}(g)\:\}.
\]
But any fixed point of $g^{n}$ gives rise to a fixed point of $g^{n\lceil n_{0}(g)/n\rceil}$
so this means $g^{n}$ has fewer than $Mp$ fixed points.

For given $N$, this implies that the number of points in $\X_{\kappa}(\FF_{p})$
fixed by any $g^{n}$ with $n\leq N$ is 
\[
\leq\sum_{n<n_{0}}Mp+\sum_{n_{0}\leq n\leq N}8Cp|\lambda|^{n}\leq n_{0}Mp+C'p|\lambda|^{N}.
\]
 for $C'=C'(g)>0$. We have $|\X_{\kappa}(\FF_{p})|\leq cp^{2}$ with
$c$ depending only on the complexity of $\X_{\kappa}$ viewed as
a variety over $\FF_{p}$, hence independent of $\kappa$. This follows
from the Lang-Weil bound \cite[Lemma 1]{LW}, and also from direct
consideration of (\ref{eq:Xkdef}). Therefore if $n_{0}Mp+C'p|\lambda|^{N}<cp^{2}$
then there exists a point in $\X_{\kappa}(\FF_{p})$ not fixed by
$g^{n}$ for any $n\leq N$. Hence there is a cycle of $g$ of length
$\geq N$ where 
\[
N\approx\frac{\log\left(\frac{cp}{2C'}-\frac{n_{0}M}{C'}\right)}{\log|\lambda|}=\frac{\log p}{\log|\lambda|}+O_{g}(1).
\]
\end{proof}

\section{Algebraic Characterization of Cyclic Palindromes\label{sec:Algebraic-Characterization-of}}

In this section, we use that $\PSL_{2}(\Z)\cong\Z/2\Z*\Z/3\Z$ with
the generators of the cyclic factors given by
\[
S=\left(\begin{array}{cc}
0 & 1\\
-1 & 0
\end{array}\right),\quad R=\left(\begin{array}{cc}
0 & 1\\
-1 & -1
\end{array}\right).
\]
Here $R=ST$ where $T=\left(\begin{array}{cc}
1 & 1\\
0 & 1
\end{array}\right)$. With this presentation, every conjugacy class in $S$ has a representative
of the form either $g=R^{y}$, $g=S$, or
\begin{equation}
g=SR^{y_{1}}\ldots SR^{y_{k}}\label{eq:cyclically-reduced}
\end{equation}
 with $y_{i}\in\{1,2\}$ for $1\leq i\leq k$. However, powers of
$S$ and $R$ are not hyperbolic, so every hyperbolic conjugacy class
has a representative as in (\ref{eq:cyclically-reduced}). Moreover,
a representative of this form has unique sequence $y_{1},\ldots,y_{k}$,
up to cyclic rotation. We write $[y_{1},\ldots,y_{k}]$ for the cyclic
equivalence class of this sequence.

\begin{proof}[Proof of Lemma \ref{lem:UV forms}]
Note that in $\PSL_{2}(\Z)$,
\begin{equation}
SR=\left(\begin{array}{cc}
-1 & -1\\
0 & -1
\end{array}\right)=V,\quad SR^{2}=\left(\begin{array}{cc}
1 & 0\\
1 & 1
\end{array}\right)=U,\label{eq:UVsubs}
\end{equation}
and substituting this into (\ref{eq:cyclically-reduced}) proves Lemma
\ref{lem:UV forms}.
\end{proof}
\begin{proof}[Proof of Proposition \ref{prop:two-defs}]
Suppose $g=U^{n_{1}}V^{m_{1}}\ldots U^{n_{l}}V^{m_{l}}$ is a hyperbolic
reduced $UV$-word. Also suppose $g$ is given by (\ref{eq:cyclically-reduced}).
Then $g^{-1}$ is conjugate in $\PSL_{2}(\Z)$ to
\begin{equation}
SR^{(1-y_{k})}SR^{(1-y_{k-1})}\ldots SR^{(1-y_{1})}.\label{eq:exp1}
\end{equation}
The action of $\PGL_{2}(\Z)$ on conjugacy classes in $\PSL_{2}(\Z)$
is generated by
\[
w=\left(\begin{array}{cc}
0 & 1\\
1 & 0
\end{array}\right).
\]
We calculate
\[
wSw^{-1}=S,\quad wRw^{-1}=R^{2}.
\]
Therefore with $g$ as in (\ref{eq:cyclically-reduced}), we have
\begin{equation}
wgw^{-1}=SR^{(1-y_{1})}S\ldots SR^{(1-y_{k})}.\label{eq:exp2}
\end{equation}
which is conjugate in $\PSL_{2}(\Z)$ to $g^{-1}$. Then comparing
(\ref{eq:exp1}) and (\ref{eq:exp2}) we have that $g$ is ambiguous
if and only if $[(1-y_{1}),(1-y_{2}),\ldots,(1-y_{k})]=[(1-y_{k}),(1-y_{k-1})\ldots(1-y_{1})]$
which is if and only if $[y_{1},y_{2},\ldots,y_{k}]=[y_{k},\ldots,y_{1}]$,
and it is easy to see, using the substitutions (\ref{eq:UVsubs}),
that this happens if and only if the reduced $UV$-word giving $g$
is a cyclic palindrome.
\end{proof}

\begin{multicols}{2}
\noindent Alois Cerbu, \\ Department of Mathematics,\\ Yale University, \\ New Haven CT 06511, U.S.A. \\ {\tt alois.cerbu@yale.edu}\\ \\
\noindent Elijah Gunther, \\ Department of Mathematics,\\ Yale University, \\ New Haven CT 06511, U.S.A. \\ {\tt elijah.gunther@yale.edu}\\ \\

\columnbreak
\noindent Michael Magee, \\Department of Mathematical Sciences, \\ Durham University,\\  Durham DH1 3LE, U.K. \\  {\tt michael.r.magee@durham.ac.uk}\\ \\
\noindent Luke Peilen, \\ Department of Mathematics,\\ Yale University, \\ New Haven CT 06511, U.S.A. \\ {\tt luke.peilen@yale.edu}\\ \\
\end{multicols}
\end{document}